\documentclass[12pt,oneside,onecolumn,reqno]{amsart}

\usepackage{amsmath}
\usepackage[dvips]{epsfig}
\usepackage{amsgen, amstext,amsbsy,amsopn, amsthm, amsfonts,amssymb,amscd,euscript,enumerate,url,verbatim,calc,oldgerm}
\usepackage{amsaddr}
\usepackage{bm}
\usepackage{latexsym}
\usepackage{color}
\usepackage{wrapfig}
\usepackage{graphicx}
\usepackage{epstopdf}
\usepackage{pifont}
\usepackage{setspace}
\usepackage{float}
\usepackage{filecontents}
\usepackage{blindtext}
\usepackage{csquotes}
\usepackage{cite}
\usepackage{xcolor}%
\usepackage{textcomp}%
\usepackage{manyfoot}%
\usepackage{booktabs}%
\usepackage{algorithm}%
\usepackage{algorithmicx}%
\usepackage{algpseudocode}%
\usepackage{listings}%
\usepackage{tikz}
\usepackage{placeins}
\usepackage{float}
\usepackage{enumitem}
\usepackage[caption = false]{subfig}
\usepackage{tikz}
\tikzset{->-/.style={decoration={
			markings,
			mark=at position #1 with {\arrow{>}}},postaction={decorate}}}
\usetikzlibrary{shapes,backgrounds,calc}
\usepackage{graphics}

\theoremstyle{plain}

\theoremstyle{definition}

\theoremstyle{remark}

\newtheorem{theorem}{Theorem}[section]

\newtheorem{corollary}[theorem]{Corollary}

\theoremstyle{definition}

\newtheorem{remark}[theorem]{Remark}

\newtheorem{example}[theorem]{Example}
\theoremstyle{remark}

\makeatletter

\newcommand{\Rmnum}[1]{\expandafter\@slowromancap\romannumeral #1@}
\makeatother

\begin{document}
	\title{Local and Global Spectral Bounds for Hermitian $A_\alpha$-Matrices}
	\author{Ravinder Kumar \and   Amisha Shekhawat}
	\address{Department of Mathematics and Computing,\\Dr B R Ambedkar National Institute of Technology Jalandhar,\\ Punjab-144008, India.}
	\email{thakurrk@nitj.ac.in; amishashekhawat02@gmail.com}
	
	\begin{abstract}
We establish local and global spectral bounds for Hermitian
$A_\alpha$-matrices of mixed graphs. Using the first three spectral moments,
we obtain an upper bound for the largest eigenvalue as the largest real zero
of an explicit cubic polynomial. Vertexwise estimates for the extreme
eigenvalues yield new lower bounds for the spectral spread, including a bound
that strictly improves an existing degree-based estimate for ordinary graphs.
A local two-dimensional compression produces a spread bound involving
neighbour-degree data and gain-weighted triangles. This bound is exact for
every mixed orientation of a star and is independent of a known Zagreb-index
bound. We also derive two complementary upper bounds for sums of the smallest
eigenvalues. As further consequences, the spectral estimates provide a
computable convergence guarantee for Richardson graph filtering and a
stability certificate for residual graph-neural-network layers on directed
networks. Numerical examples illustrate the sharpness and mutual
incomparability of the proposed bounds.
	\end{abstract}
	
	\thanks{$^{\dag}$ AMS classification: Primary 05C50, 15A18; Secondary 15A42, 65F10, 68T07}
	\maketitle 
\section{Introduction}\label{sec1}

Spectral graph theory relates structural properties of a graph to the eigenvalues of matrices associated with it. The adjacency matrix of an undirected graph is real symmetric, whereas the usual adjacency matrix of a digraph need not have a real spectrum. Guo and Mohar \cite{guomohar} introduced a Hermitian adjacency matrix for digraphs and mixed graphs that records edge orientations while retaining a real spectrum.

Let $G$ be a mixed graph with vertex set $V(G)=\{v_1,v_2,\ldots,v_n\}$. Its Hermitian adjacency matrix $H(G)=(h_{ij})$ is defined by
\[
h_{ij} = \begin{cases} 
	1 & \text{if }v_iv_j\text{ is an undirected edge or a digon}, \\
	i & \text{if }v_iv_j\text{ is an arc but }v_jv_i\text{ is not}, \\
	-i & \text{if }v_jv_i\text{ is an arc but }v_iv_j\text{ is not}, \\
	0 & \text{otherwise,}
\end{cases}
\]
where $i=\sqrt{-1}$. Thus, $H(G)$ is Hermitian, $h_{ii}=0$, and $|h_{ij}|\in\{0,1\}$ for $i\ne j$. Throughout the paper, $m$ denotes the number of adjacent unordered vertex pairs, that is, the number of edges in the underlying simple graph of $G$.

Nikiforov \cite{nikiforov} introduced the matrix $A_\alpha(G)=\alpha D(G)+(1-\alpha)A(G)$ for simple graphs. In the Hermitian setting, we consider
\[
A_\alpha(G) = \alpha D(G) + (1-\alpha)H(G)
\]
for $0\leq\alpha\leq1$, where $D(G)=\operatorname{diag}(d_1,\ldots,d_n)$ and $d_k=\sum_j|h_{kj}|^2$. Thus, $A_0(G)=H(G)$, $2A_{1/2}(G)=D(G)+H(G)$, and $A_1(G)=D(G)$.

The $A_\alpha$-spread and sums of $A_\alpha$-eigenvalues have been studied for undirected graphs; see, for example, \cite{linspread,linlaa,linsums}. This motivates estimates in the Hermitian setting that depend on accessible structural information such as degrees and gain-weighted triangle contributions.

The principal novelty of this work is a local two-dimensional compression
method for Hermitian $A_\alpha$-matrices. For each non-isolated vertex
$v_i$, we compress $A_\alpha$ to the subspace
\[
\operatorname{span}\left\{e_i,\frac{He_i}{\sqrt{d_i}}\right\}.
\]
The resulting spread bound incorporates the degrees of the neighbours of
$v_i$ and its local gain-weighted triangle contribution, in contrast to
global estimates based only on the order, size, degree range, or Zagreb
indices. It is exact for every mixed orientation of a star and is independent
of the Zagreb-index bound in \cite{linlaa}. Thus, the compression captures
local spectral information that is not contained in the cited global
estimates.

We also use the first three spectral moments to bound the largest eigenvalue
by the largest real zero of an explicit cubic polynomial. Vertexwise estimates
for both extreme eigenvalues produce further spread bounds; for ordinary
connected graphs without isolated vertices, one of them strictly improves the
estimate in \cite[Theorem~4.1]{linspread}. In addition, Cauchy interlacing and
principal-submatrix compressions yield two upper bounds for sums of the
smallest eigenvalues. Finally, motivated by the use of Hermitian graph
operators in directed graph signal processing and spectral graph neural
networks \cite{furutani,magnet}, we apply our estimates to Richardson graph
filtering and residual graph-learning layers. The cubic moment bound provides
a computable stable step size, a convergence certificate, and a layerwise
Lipschitz bound, whereas the local spread bound gives a lower bound on the
optimal contraction factor. For the directed star $\vec S_6$, this lower bound
is $0.7391$, while the exact optimal contraction factor is $0.7392$. The
examples also show that several of the new and existing bounds are mutually
incomparable.

\section{Local and global spectral bounds}\label{sec2}
Let $G$ be a mixed graph on $n$ vertices, with Hermitian adjacency matrix $H(G)$ and degree matrix $D(G)$, where $d_i=\sum_{j=1}^n|h_{ij}|^2$. For $0\leq\alpha\leq1$, set
\[
A_\alpha(G) = \alpha D(G) + (1-\alpha)H(G).
\]
Write $Z_1(G)=\sum_i d_i^2$, $F(G)=\sum_i d_i^3$, and $t_H(G)=\frac16\operatorname{tr}(H(G)^3)$. The quantity $t_H(G)$ is the real gain-weighted triangle contribution. Direct calculation gives
\begin{align}
	\operatorname{tr}\left( A_\alpha\right) &= 2\alpha m =: s_1, \label{trace1}\\
	\operatorname{tr}\left( A_\alpha^2\right) &= \alpha^2 Z_1(G) + 2(1-\alpha)^2 m =: s_2, \label{trace2}\\
	\operatorname{tr}\left( A_\alpha^3\right) &= \alpha^3 F(G) + 3\alpha(1-\alpha)^2 Z_1(G) + 6(1-\alpha)^3 t_H(G) =: s_3. \label{trace3}
\end{align}
Here we used $\operatorname{tr}(H(G)) = 0$, $\operatorname{tr}(H(G)^2) = 2m$, and $\operatorname{tr}(D(G)H(G)^2) = Z_1(G)$.

\begin{theorem}\label{Thm1}
	Let $G$ be a connected mixed graph with $n\geq2$ vertices and $m$ edges in its underlying graph, and let $\mu_1\leq\cdots\leq\mu_n$ be the eigenvalues of $A_\alpha(G)$. If $\frac12<\alpha\leq1$, then
	\[
	\mu_n(A_\alpha) \leq r_{\max},
	\]
	where $r_{\max}$ is the largest real zero of
	\begin{equation}\label{eq4}
		P(x) = s_1 x^3 - 2s_2 x^2 + s_3 x + s_2^2 - s_1 s_3,
	\end{equation}
	and $s_1, s_2, s_3$ are defined in \eqref{trace1}--\eqref{trace3}.
\end{theorem}

\begin{proof}
	Since $H(G)$ is Hermitian, so is $A_\alpha(G)$. Moreover,
	\[
	A_\alpha = (2\alpha - 1)D + (1 - \alpha)(D + H).
	\]
	Because $G$ is connected and $n\geq2$, every $d_i$ is positive, so $D$ is positive definite. Furthermore, $D+H$ is positive semidefinite; indeed, for every $z\in\mathbb{C}^n$,
	\[
	z^*(D+H)z=\sum_{\{i,j\}\in E(G)}|z_i+h_{ij}z_j|^2\geq0.
	\]
	Hence, $A_\alpha$ is positive definite for $\alpha>\frac12$, and $\mu_i>0$ for every $i$.
	
	Fix $j\in\{1,\ldots,n\}$. Applying the Cauchy--Schwarz inequality to $(\mu_i^{3/2})_{i\neq j}$ and $(\mu_i^{1/2})_{i\neq j}$ yields
	\begin{equation}\label{eq5}
		\left(\sum_{i \neq j} \mu_i^2 \right)^2 \leq \left(\sum_{i \neq j} \mu_i^3 \right) \left(\sum_{i \neq j} \mu_i \right).
	\end{equation}
	By \eqref{trace1}--\eqref{trace3}, this is
	\[
	(s_2 - \mu_j^2)^2 \leq (s_3 - \mu_j^3)(s_1 - \mu_j),
	\]
	which, after expansion and cancellation, is equivalent to $P(\mu_j) \leq 0$. Since $s_1 = 2\alpha m > 0$, the cubic $P(x)$ is positive to the right of its largest real zero. Therefore, $\mu_j \leq r_{\max}$ for every $j$, and the claim follows by choosing $j = n$.
\end{proof}

\begin{remark}\label{rem:eqthm1}
	For $n\geq3$, equality in Theorem~\ref{Thm1} holds if and only if
	$\mu_1=\cdots=\mu_{n-1}$. This is precisely the equality condition in the Cauchy--Schwarz step with $j=n$. For $n=2$, equality always holds.
\end{remark}

\begin{example}\label{ex1}
	Consider the wheel graph $G=\vec{W}_7$ with $n=7$ vertices and $m=12$
	edges in its underlying graph, oriented as shown in Figure~\ref{fig:graph_ex1}.
	
	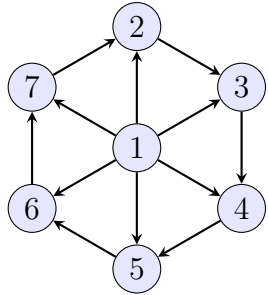
\begin{figure}[H]
		\centering
		\begin{tikzpicture}[scale=0.8, auto, swap, >=stealth]
			\node[circle, draw, fill=blue!10, inner sep=0pt,
			minimum size=18pt] (V1) at (0,0) {$1$};
			\node[circle, draw, fill=blue!10, inner sep=0pt,
			minimum size=18pt] (V2) at (90:2.0cm) {$2$};
			\node[circle, draw, fill=blue!10, inner sep=0pt,
			minimum size=18pt] (V3) at (30:2.0cm) {$3$};
			\node[circle, draw, fill=blue!10, inner sep=0pt,
			minimum size=18pt] (V4) at (330:2.0cm) {$4$};
			\node[circle, draw, fill=blue!10, inner sep=0pt,
			minimum size=18pt] (V5) at (270:2.0cm) {$5$};
			\node[circle, draw, fill=blue!10, inner sep=0pt,
			minimum size=18pt] (V6) at (210:2.0cm) {$6$};
			\node[circle, draw, fill=blue!10, inner sep=0pt,
			minimum size=18pt] (V7) at (150:2.0cm) {$7$};
			
			\draw[->, thick] (V1) -- (V2);
			\draw[->, thick] (V1) -- (V3);
			\draw[->, thick] (V1) -- (V4);
			\draw[->, thick] (V1) -- (V5);
			\draw[->, thick] (V1) -- (V6);
			\draw[->, thick] (V1) -- (V7);
			
			\draw[->, thick] (V2) -- (V3);
			\draw[->, thick] (V3) -- (V4);
			\draw[->, thick] (V4) -- (V5);
			\draw[->, thick] (V5) -- (V6);
			\draw[->, thick] (V6) -- (V7);
			\draw[->, thick] (V7) -- (V2);
		\end{tikzpicture}
		\caption{The wheel graph $\vec{W}_7$.}
		\label{fig:graph_ex1}
	\end{figure}
	
	For this digraph, $Z_1(G)=90$, $F(G)=378$, and $t_H(G)=0$. Let $e_1$ correspond to the central vertex and put $u=\frac1{\sqrt6}(0,1,\ldots,1)^T$. The subspace spanned by $e_1$ and $u$ is invariant under $A_\alpha(\vec W_7)$, and the corresponding Hermitian matrix is
	\[
	B_\alpha = \begin{pmatrix}
		6\alpha & i\sqrt6(1-\alpha)\\
		-i\sqrt6(1-\alpha) & 3\alpha
	\end{pmatrix}.
	\]
	Its larger eigenvalue is
	\[
	\mu_n\bigl(A_\alpha(\vec{W}_7)\bigr) = \frac{9\alpha + \sqrt{9\alpha^2 + 24(1-\alpha)^2}}{2}.
	\]
	The other five eigenvalues, obtained from the nonconstant Fourier modes on the directed rim cycle, are
	\[
	3\alpha-2(1-\alpha)\sin\left(\frac{k\pi}{3}\right),\qquad k=1,\ldots,5.
	\]
	Their maximum does not exceed the larger eigenvalue of $B_\alpha$ for $0\leq\alpha\leq1$. Consequently, the displayed expression is the largest eigenvalue of $A_\alpha(\vec W_7)$. At $\alpha=1$, it equals $6$, as expected.
	
	Substituting $Z_1$, $F$, and $t_H(G)$ into \eqref{eq4} gives the values of $r_{\max}$ reported in Table~\ref{tab:bounds_comparison}.
	
	\begin{table}[h!]
		\centering
		\caption{Comparison of the exact spectral radius, the upper bound, and the corresponding absolute error for $\vec{W}_7$.}
		\label{tab:bounds_comparison}
		\begin{tabular}{cccc}
			\toprule
			$\alpha$ & $\mu_n(A_\alpha)$ & $r_{\max}$ & $r_{\max}-\mu_n(A_\alpha)$ \\
			\midrule
			0.6 & 4.0304 & 4.1827 & 0.1523 \\
			0.7 & 4.4316 & 4.5196 & 0.0880 \\
			0.8 & 4.8961 & 4.9339 & 0.0378 \\
			0.9 & 5.4220 & 5.4308 & 0.0088 \\
			1.0 & 6.0000 & 6.0000 & 0.0000 \\
			\bottomrule
		\end{tabular}
	\end{table}
\end{example}

For the next result, let $H=(h_{ij})$ be Hermitian with zero diagonal and
$|h_{ij}|\in\{0,1\}$ for $i\ne j$. Define
\[
d_i=\sum_{j\ne i}|h_{ij}|^2,\quad
D=\operatorname{diag}(d_1,\ldots,d_n),\quad
A_\alpha=\alpha D+(1-\alpha)H,
\]
and put $\delta=\min_i d_i$ and $\Delta=\max_i d_i$. This formulation includes
ordinary graphs and graph models defined through a Hermitian adjacency matrix.

\begin{theorem}\label{Thm2}
	Assume that $\delta>0$, and arrange the eigenvalues of $A_\alpha$ as
	$\mu_1\leq\cdots\leq\mu_n$. For $0\leq\alpha<1$, set
	\[
	\ell_\alpha=
	\begin{cases}
		(2\alpha-1)\Delta,&0\leq\alpha<\frac12,\\
		(2\alpha-1)\delta,&\frac12\leq\alpha<1.
	\end{cases}
	\]
	Then
	\begin{equation}\label{eq6}
		\mu_1(A_\alpha) \le \alpha \delta - \frac{(1-\alpha)^2 \delta}{\Delta - \alpha \delta},
	\end{equation}
	\begin{equation}\label{eq7}
		\mu_n(A_\alpha) \ge \alpha \Delta + \frac{(1-\alpha)^2 \Delta}{\alpha \Delta - \ell_\alpha}.
	\end{equation}
	In particular, if $\frac12\leq\alpha<1$, the denominator in
	\eqref{eq7} equals $\alpha\Delta-(2\alpha-1)\delta$.
\end{theorem}
\begin{proof}
	Relabel the indices so that $d_1=\delta$ and $d_n=\Delta$. Then
	$a_{11}=\alpha\delta$ and $a_{nn}=\alpha\Delta$.
	
	Since $A_\alpha$ is Hermitian and its spectrum lies in $[\mu_1,\mu_n]$,
	functional calculus gives $(A_\alpha-\mu_1I)(\mu_nI-A_\alpha)\geq0$.
	Its $(i,i)$-entry yields
	\begin{equation}\label{eq8}
		(a_{ii}-\mu_1)(\mu_n-a_{ii})
		\geq\sum_{j\ne i}|a_{ij}|^2=(1-\alpha)^2d_i.
	\end{equation}
	For $i=1$, Gershgorin's theorem gives $\mu_n\leq\Delta$, and hence
	\[
	\alpha\delta-\mu_1\geq
	\frac{(1-\alpha)^2\delta}{\mu_n-\alpha\delta}
	\geq\frac{(1-\alpha)^2\delta}{\Delta-\alpha\delta},
	\]
	which proves \eqref{eq6}.
	
	For $i=n$, \eqref{eq8} gives
	\[
	\mu_n-\alpha\Delta\geq
	\frac{(1-\alpha)^2\Delta}{\alpha\Delta-\mu_1}.
	\]
	Gershgorin's theorem also yields
	$\mu_1\geq\min_i(2\alpha-1)d_i=\ell_\alpha$. Consequently,
	$\alpha\Delta-\mu_1\leq\alpha\Delta-\ell_\alpha$, and
	\eqref{eq7} follows.
\end{proof}

\begin{remark}\label{rem:eqthm2}
	Equality in the diagonal inequality \eqref{eq8} for an index $i$ holds precisely when
	\[
	(A_\alpha-\mu_1I)(\mu_nI-A_\alpha)e_i=0.
	\]
	Thus, $e_i$ has spectral support only in the eigenspaces corresponding to $\mu_1$ and $\mu_n$. Equality in \eqref{eq6} additionally requires $\mu_n=\Delta$, while equality in \eqref{eq7} additionally requires $\mu_1=\ell_\alpha$.
\end{remark}

\begin{theorem}\label{Thm2vertex}
	Under the assumptions of Theorem~\ref{Thm2}, for each $i=1,\ldots,n$,
	\begin{align}
		\mu_1(A_\alpha)&\leq \alpha d_i-
		\frac{(1-\alpha)^2d_i}{\Delta-\alpha d_i},\label{eq6vertex}\\
		\mu_n(A_\alpha)&\geq \alpha d_i+
		\frac{(1-\alpha)^2d_i}{\alpha d_i-\ell_\alpha}.\label{eq7vertex}
	\end{align}
	Consequently,
	\[
	\mu_1(A_\alpha)\leq\min_i\left\{\alpha d_i-
	\frac{(1-\alpha)^2d_i}{\Delta-\alpha d_i}\right\},
	\quad
	\mu_n(A_\alpha)\geq\max_i\left\{\alpha d_i+
	\frac{(1-\alpha)^2d_i}{\alpha d_i-\ell_\alpha}\right\}.
	\]
\end{theorem}

\begin{proof}
	Apply \eqref{eq8} at an arbitrary index $i$, and use
	$\mu_n\leq\Delta$ and $\mu_1\geq\ell_\alpha$. The denominators are positive because $\delta>0$ and $0\leq\alpha<1$. Taking the minimum and maximum, respectively, completes the proof.
\end{proof}

\begin{corollary}\label{cor1}
	Under the assumptions of Theorem~\ref{Thm2}, the spectral spread $S(A_\alpha) = \mu_n(A_\alpha) - \mu_1(A_\alpha)$ satisfies
	\begin{equation}\label{eq9}
		S(A_\alpha) \ge \alpha(\Delta - \delta) + (1-\alpha)^2 \left( \frac{\Delta}{\alpha\Delta - \ell_\alpha} + \frac{\delta}{\Delta - \alpha\delta} \right).
	\end{equation}
	In particular, if $G$ is a $d$-regular mixed graph ($\delta = \Delta = d$), then for $0\leq\alpha<1$
	\begin{equation}\label{eq10}
		S(A_\alpha) \ge 2(1-\alpha).
	\end{equation}
\end{corollary}

\begin{remark}
	If $G$ is $d$-regular, then $A_\alpha=\alpha dI+(1-\alpha)H$ and hence
	$S(A_\alpha)=(1-\alpha)S(H)$. Therefore, equality in \eqref{eq10} holds exactly when $S(H)=2$; in particular, equality holds for the connected mixed graph on two vertices.
\end{remark}

The next observation records both the improvement over the earlier degree bound and the full degree-sequence consequence of Theorem~\ref{Thm2vertex}.

\begin{corollary}\label{cor:degree-refinement}
	Under the assumptions of Theorem~\ref{Thm2},
	\begin{align}
		S(A_\alpha)\geq {}&
		\max_{1\leq i\leq n}\left\{\alpha d_i+
		\frac{(1-\alpha)^2d_i}{\alpha d_i-\ell_\alpha}\right\}
		\nonumber\\[-1mm]
		&-\min_{1\leq i\leq n}\left\{\alpha d_i-
		\frac{(1-\alpha)^2d_i}{\Delta-\alpha d_i}\right\}.
		\label{eq:degree-sequence-spread}
	\end{align}
	In particular, \eqref{eq:degree-sequence-spread} is at least as strong as \eqref{eq9}.
\end{corollary}

\begin{proof}
	Subtract the upper bound for $\mu_1(A_\alpha)$ in Theorem~\ref{Thm2vertex}
	from its lower bound for $\mu_n(A_\alpha)$. Choosing an index of maximum
	degree in the first extremum and an index of minimum degree in the second
	extremum gives \eqref{eq9}.
\end{proof}

\begin{remark}\label{rem:strict-improvement}
	For an ordinary connected graph with $\delta>0$, inequality \eqref{eq9}
	strictly improves the degree-based lower bound in \cite[Theorem~4.1]{linspread}.
	Indeed, if $0\leq\alpha<\frac12$, then $\ell_\alpha=(2\alpha-1)\Delta$ and
	\[
	(1-\alpha)^2\left(
	\frac{\Delta}{\alpha\Delta-\ell_\alpha}
	+\frac{\delta}{\Delta-\alpha\delta}\right)
	=(1-\alpha)+\frac{(1-\alpha)^2\delta}{\Delta-\alpha\delta}>\alpha.
	\]
	If $\frac12\leq\alpha<1$, then $\ell_\alpha=(2\alpha-1)\delta$ and
	\[
	\frac{\Delta}{\alpha\Delta-(2\alpha-1)\delta}
	+\frac{\delta}{\Delta-\alpha\delta}>\frac1\alpha.
	\]
	These inequalities give a strict improvement in both parameter ranges.
\end{remark}

For a vertex $v_i$, put
\[
q_i=\sum_{j=1}^n d_j|h_{ji}|^2,
\qquad \tau_i=(H^3)_{ii}.
\]
The number $q_i$ is the sum of the degrees of the neighbours of $v_i$, while
$\tau_i$ is a real local gain-weighted triangle contribution. For an ordinary
graph, $\tau_i=2t_i$, where $t_i$ is the number of triangles containing $v_i$.

\begin{theorem}\label{thm:local-compression-spread}
	Let $G$ be a mixed graph and let $0\leq\alpha<1$. For every vertex $v_i$
	with $d_i>0$, define
	\[
	\beta_i=\frac{\alpha q_i+(1-\alpha)\tau_i}{d_i}.
	\]
	Then
	\begin{equation}\label{eq:local-compression-spread}
		S(A_\alpha)\geq
		\max_{d_i>0}\sqrt{(\alpha d_i-\beta_i)^2+4(1-\alpha)^2d_i}.
	\end{equation}
	Consequently,
	\begin{equation}\label{eq:max-degree-spread}
		S(A_\alpha)\geq2(1-\alpha)\sqrt{\Delta}.
	\end{equation}
\end{theorem}

\begin{proof}
	Fix $i$ with $d_i>0$ and set
	\[
	x=e_i,\qquad y=\frac{He_i}{\sqrt{d_i}}.
	\]
	Since $h_{ii}=0$ and $\|He_i\|^2=(H^2)_{ii}=d_i$, the vectors $x$ and
	$y$ are orthonormal. The compression of $A_\alpha$ to
	$\operatorname{span}\{x,y\}$ is
	\[
	B_i=
	\begin{pmatrix}
		\alpha d_i&(1-\alpha)\sqrt{d_i}\\
		(1-\alpha)\sqrt{d_i}&\beta_i
	\end{pmatrix}.
	\]
	Indeed,
	\[
	y^*A_\alpha y
	=\frac{\alpha e_i^*HDHe_i+(1-\alpha)e_i^*H^3e_i}{d_i}
	=\frac{\alpha q_i+(1-\alpha)\tau_i}{d_i}=\beta_i.
	\]
	The difference between the two eigenvalues of $B_i$ is
	\[
	\sqrt{(\alpha d_i-\beta_i)^2+4(1-\alpha)^2d_i}.
	\]
	By the min--max principle, the two eigenvalues of the compression lie
	between the least and largest eigenvalues of $A_\alpha$. Hence their
	difference is at most $S(A_\alpha)$, which proves
	\eqref{eq:local-compression-spread}.
	Dropping the nonnegative square and choosing a vertex of maximum degree
	gives \eqref{eq:max-degree-spread}.
\end{proof}

\begin{remark}
	Equality in the bound corresponding to a vertex $v_i$ in
	\eqref{eq:local-compression-spread} holds if and only if the two Ritz values
	of $B_i$ are $\mu_1(A_\alpha)$ and $\mu_n(A_\alpha)$. Thus equality holds
	whenever $\operatorname{span}\{e_i,He_i\}$ carries both extreme eigenvalues.
\end{remark}

\begin{corollary}\label{cor:mixed-star-exact}
	Let $G$ be any mixed orientation of the star $K_{1,n-1}$, where $n\geq2$,
	and let $0\leq\alpha<1$. Then equality holds in
	\eqref{eq:local-compression-spread}. More precisely,
	\[
	S(A_\alpha(G))
	=\sqrt{\alpha^2(n-2)^2+4(1-\alpha)^2(n-1)}.
	\]
\end{corollary}

\begin{proof}
	Since a star is a tree, a diagonal unitary switching removes all edge
	phases from its Hermitian adjacency matrix. This switching leaves the
	degree matrix unchanged. Hence $A_\alpha(G)$ is unitarily similar to
	$A_\alpha(K_{1,n-1})$. The latter matrix has the eigenvalue $\alpha$ with
	multiplicity $n-2$, and its remaining two eigenvalues are
	\[
	\lambda_\pm
	=\frac{\alpha n\pm
		\sqrt{\alpha^2(n-2)^2+4(1-\alpha)^2(n-1)}}{2}.
	\]
	The square-root term is at least $\alpha(n-2)$, so
	$\lambda_-\leq\alpha\leq\lambda_+$. Therefore,
	\[
	S(A_\alpha(G))=\lambda_+-\lambda_-
	=\sqrt{\alpha^2(n-2)^2+4(1-\alpha)^2(n-1)}.
	\]
	At the central vertex, $d_i=n-1$, $q_i=n-1$, and $\tau_i=0$, so
	$\beta_i=\alpha$. Substitution in
	\eqref{eq:local-compression-spread} gives exactly the same quantity.
\end{proof}

\begin{example}\label{ex:spread-comparison}
	The local bound \eqref{eq:local-compression-spread} and the Zagreb-index
	bound in \cite[Theorem~3.6]{linlaa} are independent. First consider the
	ordinary star $K_{1,5}$ and let $\alpha=0.7$. At its central vertex,
	$d_i=5$, $q_i=5$, and $\tau_i=0$, so that $\beta_i=0.7$. Hence
$
	S(A_{0.7})\geq
	\sqrt{(3.5-0.7)^2+4(0.3)^2(5)}
	=\sqrt{9.64}=3.1048.
	$
	The eigenvalues of $A_\alpha(K_{1,5})$ show that
	$S(A_{0.7})=3.1048$, so the new bound is exact. For comparison, the earlier
	degree bound in \cite[Theorem~4.1]{linspread}, inequality \eqref{eq9}, and
	the Zagreb-index bound in \cite[Theorem~3.6]{linlaa} give, respectively,
	$2.9286$, $2.9661$, and $2.2261$.
	
	Conversely, for the cycle $C_5$ at $\alpha=0.7$, every vertex satisfies
	$d_i=2$, $q_i=4$, and $\tau_i=0$. Thus
	\eqref{eq:local-compression-spread} gives $0.8485$, whereas the
	Zagreb-index bound gives $0.8660$. Therefore neither bound dominates the
	other in general.
\end{example}

\begin{theorem}\label{Thm3}
	Let $G$ be a mixed graph on $n$ vertices, let $0\leq\alpha\leq1$, and let $K=G[S]$ be the induced mixed subgraph on
	$S=\{v_1,\ldots,v_r,v_t\}$, where $1\leq r\leq n-1$. Put
	$K_t=K-v_t$ and
	\[
	L_t=\min_{1\leq i\leq r}
	\{\alpha d_G(v_i)-(1-\alpha)d_{K_t}(v_i)\}.
	\]
	Then
	\begin{align}\label{eq11}
		\sum_{i=1}^{r}\mu_i(A_\alpha(G))
		\leq {}&\alpha\sum_{i=1}^{r}d_G(v_i)\nonumber\\
		&-\frac{\sqrt{(\alpha d_G(v_t)-L_t)^2+
				4(1-\alpha)^2d_K(v_t)}-(\alpha d_G(v_t)-L_t)}{2}.
	\end{align}
\end{theorem}

\begin{proof}
	Write the principal submatrix $M=A_\alpha(G)[S]$ in block form as
	\[
	M=\begin{pmatrix}C&b\\b^*&a\end{pmatrix},
	\qquad a=\alpha d_G(v_t).
	\]
	Here $C=A_\alpha(G)[S\setminus\{v_t\}]$ and
	$\|b\|^2=(1-\alpha)^2d_K(v_t)$. If $b=0$, then
	$\mu_{r+1}(M)\geq\max\{a,L_t\}$, which gives \eqref{eq11}. Suppose that $b\ne0$ and put $u=b/\|b\|$. The compression of $M$ to the span of $(u,0)^T$ and $(0,1)^T$ is
	\[
	\begin{pmatrix}u^*Cu&\|b\|\\\|b\|&a\end{pmatrix}.
	\]
	Gershgorin's theorem gives $\lambda_{\min}(C)\geq L_t$, so $u^*Cu\geq L_t$. The largest eigenvalue of the displayed compression is increasing in its first diagonal entry. Hence,
	\[
	\mu_{r+1}(M)\geq
	\frac{a+L_t+\sqrt{(a-L_t)^2+4\|b\|^2}}{2}.
	\]
	Since
	$\operatorname{tr}(M)=\alpha\sum_{i=1}^r d_G(v_i)+a$,
	we obtain an upper bound for the sum of the smallest $r$ eigenvalues of $M$. Finally, Cauchy's interlacing theorem gives
	$\sum_{i=1}^r\mu_i(A_\alpha(G))\leq\sum_{i=1}^r\mu_i(M)$, which proves \eqref{eq11}.
\end{proof}

\begin{remark}\label{rem:eqthm3}
	When $b\ne0$, equality in \eqref{eq11} holds precisely when equality occurs in each of the three estimates used in its proof: $u^*Cu=L_t$, the largest eigenvalue of the two-dimensional compression equals $\mu_{r+1}(M)$, and
	$\sum_{i=1}^r\mu_i(A_\alpha(G))=\sum_{i=1}^r\mu_i(M)$.
\end{remark}

\begin{example}\label{ex2}
	Consider the directed cycle graph $G = C_4$ with $n = 4$ vertices and $m = 4$ edges, as shown in Figure~\ref{fig:c4_graph}.
	
	\begin{figure}[H]
		\centering
		\begin{tikzpicture}[scale=0.8, auto, swap]
			\node[circle, draw, fill=blue!10, inner sep=0pt, minimum size=18pt] (V1) at (0,1.5) {$v_1$};
			\node[circle, draw, fill=blue!10, inner sep=0pt, minimum size=18pt] (V2) at (1.5,1.5) {$v_2$};
			\node[circle, draw, fill=blue!10, inner sep=0pt, minimum size=18pt] (V3) at (1.5,0) {$v_3$};
			\node[circle, draw, fill=blue!10, inner sep=0pt, minimum size=18pt] (V4) at (0,0) {$v_4$};
			
			\draw[thick, ->, >=stealth] (V1) -- (V2);
			\draw[thick, ->, >=stealth] (V2) -- (V3);
			\draw[thick, ->, >=stealth] (V3) -- (V4);
			\draw[thick, ->, >=stealth] (V4) -- (V1);
		\end{tikzpicture}
		\caption{The directed cycle graph $C_4$.}
		\label{fig:c4_graph}
	\end{figure}
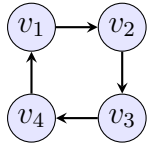
	
	Let $\alpha=1/2$. The matrix $A_{1/2}(C_4)$ is
	\[
	A_{1/2}(C_4) = 
	\begin{pmatrix}
		1 & i/2 & 0 & -i/2 \\
		-i/2 & 1 & i/2 & 0 \\
		0 & -i/2 & 1 & i/2 \\
		i/2 & 0 & -i/2 & 1
	\end{pmatrix}
	\]
	with eigenvalues $\mu_1 = 0$, $\mu_2 = 1$, $\mu_3 = 1$, and $\mu_4 = 2$. 
	
	Applying Theorem~\ref{Thm3} for induced subgraphs $H = C_4[S]$ with target vertex $v_t$, we evaluate the upper bounds for the eigenvalue sum $\sum_{i=1}^r \mu_i(A_{1/2}(C_4))$ across subset sizes $r = 1, 2, 3$. The exact eigenvalue sums, theoretical upper bounds, and corresponding absolute errors are reported in Table~\ref{tab:c4_bounds_comparison}.
	
	\begin{table}[h!]
		\centering
		\caption{Comparison of the exact eigenvalue sum $\sum_{i=1}^r \mu_i$, the upper bounds, and their absolute errors for $C_4$ at $\alpha=1/2$.}
		\label{tab:c4_bounds_comparison}
		\begin{tabular}{ccccccc}
			\toprule
			$r$ & $S$ & $v_t$ & Exact $\sum_{i=1}^r \mu_i$ & Upper Bound & Absolute Error \\
			\midrule
			1 & $\{v_1, v_2\}$ & $v_2$ & 0.0000 & 0.5000 & 0.5000 \\
			2 & $\{v_1, v_2, v_3\}$ & $v_3$ & 1.0000 & 1.6910 & 0.6910 \\
			3 & $V(C_4)$ & $v_4$ & 2.0000 & 2.6340 & 0.6340 \\
			\bottomrule
		\end{tabular}
	\end{table}
	\FloatBarrier
\end{example}

\begin{theorem}\label{Thm4}
	Let $G$ be a mixed graph with $n$ vertices, let $0\leq\alpha\leq1$, and let $K=G[S]$ be an induced mixed subgraph on $p=r+1$ vertices with vertex set $S=\{v_1,v_2,\ldots,v_r,v_t\}$, where $1\leq r\leq n-1$. For any $v_k\in\{v_1,\ldots,v_r\}$,
	\begin{align}\label{eq13}
		\sum_{i=1}^{r}\mu_i(A_\alpha(G))
		\leq {}&\alpha\sum_{i=1}^{r}d_G(v_i)\nonumber\\
		&-\frac{\sqrt{\alpha^2d_{tk}^2+
				4(1-\alpha)^2|h_{tk}|^2}-\alpha d_{tk}}{2},
	\end{align}
	where $d_{tk} = d_G(v_t) - d_G(v_k)$.
\end{theorem}

\begin{proof}
	Let $M=A_\alpha(G)[S]$ be the principal submatrix indexed by $S$.
	The principal submatrix of $M$ indexed by $\{v_k,v_t\}$ is
	\[
	B=\begin{bmatrix}
		\alpha d_G(v_k)&(1-\alpha)h_{kt}\\
		(1-\alpha)h_{tk}&\alpha d_G(v_t)
	\end{bmatrix}.
	\]
	Its largest eigenvalue is
	\[
	\mu_{\max}(B) = \alpha d_G(v_t)
	+\frac{\sqrt{\alpha^2 d_{tk}^2+4(1-\alpha)^2|h_{tk}|^2}
		-\alpha d_{tk}}{2}.
	\]
	By Cauchy's interlacing theorem, $\mu_p(M)\geq\mu_{\max}(B)$. Moreover,
	\[
	\operatorname{tr}(M)
	=\alpha\sum_{i=1}^r d_G(v_i)+\alpha d_G(v_t).
	\]
	Therefore,
	\begin{align*}
		\sum_{i=1}^{r}\mu_i(M)
		&=\operatorname{tr}(M)-\mu_p(M)\\
		&\leq\alpha\sum_{i=1}^{r}d_G(v_i)+\alpha d_G(v_t)-\mu_{\max}(B).
	\end{align*}
	Applying Cauchy's interlacing principle again,
	\[
	\sum_{i=1}^r\mu_i(A_\alpha(G))
	\leq\sum_{i=1}^r\mu_i(M),
	\]
	which gives \eqref{eq13}.
\end{proof}

\begin{remark}\label{rem:eqthm4}
	Equality in \eqref{eq13} holds precisely when the largest eigenvalue of the
	$2\times2$ principal submatrix $B$ equals the largest eigenvalue of
	$A_\alpha(G)[S]$ and equality holds in the sum form of the interlacing
	relation between $A_\alpha(G)$ and $A_\alpha(G)[S]$.
\end{remark}

\begin{corollary}\label{cor2}
	Under the hypotheses of Theorem~\ref{Thm4}, maximize over
	$k=1,\dots,r$. The strongest bound in this family is
	\begin{align}\label{eq14}
		\sum_{i=1}^{r}\mu_i(A_\alpha(G))
		\leq {}&\alpha\sum_{i=1}^{r}d_G(v_i)\nonumber\\
		&-\max_{1\leq k\leq r}
		\left\{\frac{\sqrt{\alpha^2d_{tk}^2+
				4(1-\alpha)^2|h_{tk}|^2}-\alpha d_{tk}}{2}\right\}.
	\end{align}
\end{corollary}

\begin{example}\label{ex3}
	Let $A_\alpha(G) = \alpha D(G) + (1-\alpha) H(G)$, where $D(G)$ is the diagonal matrix of vertex degrees and $H(G)$ is the Hermitian adjacency matrix of $G$. To compare the bounds obtained in Theorem~\ref{Thm3} and Theorem~\ref{Thm4}, consider the directed path $\vec{P}_4$ and the star digraph $\vec{S}_4$ on $n = 4$ vertices shown in Figure~\ref{fig:digraphs}.
	
	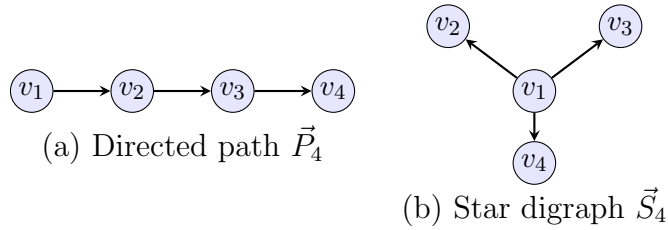
\begin{figure}[h!]
		\centering
		\begin{tikzpicture}[scale=0.95, >=stealth, main/.style={circle, draw, fill=blue!10, inner sep=0pt, minimum size=16pt, font=\small}]
			\node[main] (p1) at (0,0) {$v_1$};
			\node[main] (p2) at (1.4,0) {$v_2$};
			\node[main] (p3) at (2.8,0) {$v_3$};
			\node[main] (p4) at (4.2,0) {$v_4$};
			
			\draw[->, thick] (p1) -- (p2);
			\draw[->, thick] (p2) -- (p3);
			\draw[->, thick] (p3) -- (p4);
			\node at (2.1,-0.75) {(a) Directed path $\vec{P}_4$};
			
			\node[main] (s1) at (7.0,0) {$v_1$};
			\node[main] (s2) at (5.8,0.9) {$v_2$};
			\node[main] (s3) at (8.2,0.9) {$v_3$};
			\node[main] (s4) at (7.0,-1.0) {$v_4$};
			
			\draw[->, thick] (s1) -- (s2);
			\draw[->, thick] (s1) -- (s3);
			\draw[->, thick] (s1) -- (s4);
			\node at (7.0,-1.6) {(b) Star digraph $\vec{S}_4$};
		\end{tikzpicture}
		\caption{The digraphs $\vec{P}_4$ and $\vec{S}_4$.}
		\label{fig:digraphs}
	\end{figure}
	
	For $\alpha=1/2$ and $r=2$, the two smallest eigenvalues of
	$A_{1/2}(\vec P_4)$ sum to $0.2929$, while those of
	$A_{1/2}(\vec S_4)$ sum to $0.5000$. We compare the two bounds as follows.
	\begin{enumerate}[label=(\roman*)]
		\item For $\vec P_4$, take $S=\{v_1,v_2,v_3\}$ and $v_t=v_1$. Inequality \eqref{eq11} gives
		\[
		\sum_{i=1}^{2}\mu_i(A_{1/2}(\vec P_4))\leq1.5000,
		\]
		whereas \eqref{eq14} gives the stronger bound $1.1910$.
		
		\item For $\vec S_4$, take $S=\{v_1,v_2,v_3\}$, where $v_1$ is the central vertex, and set $v_t=v_1$. Inequality \eqref{eq11} gives
		\[
		\sum_{i=1}^{2}\mu_i(A_{1/2}(\vec S_4))\leq0.6340,
		\]
		whereas \eqref{eq14} gives $0.7929$. Thus, Theorem~\ref{Thm3} gives the stronger estimate in this case.
	\end{enumerate}
	
	Consequently, the bounds in Theorems~\ref{Thm3} and~\ref{Thm4} are incomparable.
\end{example}

\begin{remark}\label{rem_structural_insight}
	The effectiveness of Theorems~\ref{Thm3} and~\ref{Thm4} depends on the local structure of $G[S]$:
	\begin{enumerate}[label=(\roman*)]
		\item Theorem~\ref{Thm3} uses all edges from the target vertex to $G[S]-v_t$ simultaneously and can be effective when $v_t$ has several neighbours in $G[S]$.
		\item Theorem~\ref{Thm4} uses one vertex pair and can be effective when a suitable pair produces a large two-dimensional compression eigenvalue.
	\end{enumerate}
	Hence, neither bound dominates the other in general, and the better available estimate is the minimum of the right-hand sides of \eqref{eq11} and \eqref{eq14}.
\end{remark}

\section{Applications to iterative graph filtering and directed graph learning}
\label{sec:application}

Hermitian graph operators are useful in computer science because they encode
edge directions while retaining a real spectrum and an orthonormal eigenbasis.
These properties underlie spectral filtering, representation learning, and
signal processing on directed graphs; see, for example,
\cite{furutani,magnet}. We show that the bounds developed above provide
certified parameters for an iterative graph filter, quantify an unavoidable
limitation on its best possible convergence rate, and certify the Lipschitz
stability of a residual graph-learning layer.

Let $b\in\mathbb{C}^n$ be a signal supported on the vertices of a connected
mixed graph $G$, and consider the Hermitian positive-definite system
\begin{equation}\label{eq:network-system}
	A_\alpha x=b,\qquad \frac12<\alpha<1.
\end{equation}
Equivalently, its solution is the unique minimizer of
\[
\mathcal{E}(x)=\frac12x^*A_\alpha x-\operatorname{Re}(b^*x).
\]
The Richardson iteration \cite{saad} for \eqref{eq:network-system} is
\begin{equation}\label{eq:richardson}
	x^{(k+1)}=x^{(k)}+\gamma\bigl(b-A_\alpha x^{(k)}\bigr).
\end{equation}
It is a local graph computation: at vertex $v_i$,
\begin{equation}\label{eq:node-update}
	x_i^{(k+1)}
	=x_i^{(k)}+\gamma\left(
	b_i-\alpha d_i x_i^{(k)}
	-(1-\alpha)\sum_{j=1}^n h_{ij}x_j^{(k)}
	\right).
\end{equation}
Thus each update uses only the current signal values at adjacent vertices,
with the phase of $h_{ij}$ retaining the direction of the corresponding edge.

For the following result, define the computable quantities
\begin{equation}\label{eq:application-parameters}
	\underline{\lambda}=(2\alpha-1)\delta,\qquad
	U=r_{\max},\qquad
	L=\max_{d_i>0}
	\sqrt{(\alpha d_i-\beta_i)^2+4(1-\alpha)^2d_i},
\end{equation}
where $r_{\max}$ is the largest real zero in Theorem~\ref{Thm1}, and
$\beta_i$ is defined in Theorem~\ref{thm:local-compression-spread}. All three
quantities depend only on the graph parameters used in the preceding results.

\begin{theorem}\label{thm:richardson-application}
	Let $G$ be a connected mixed graph and let $\frac12<\alpha<1$. Then:
	\begin{enumerate}[label=(\roman*)]
		\item The system \eqref{eq:network-system} has a unique solution $x^*$.
		With
		\[
		\gamma_{\rm c}=\frac{2}{\underline{\lambda}+U},
		\qquad
		q_{\rm c}=\frac{U-\underline{\lambda}}
		{U+\underline{\lambda}},
		\]
		the iteration \eqref{eq:richardson} satisfies
		\begin{equation}\label{eq:certified-rate}
			\|x^{(k)}-x^*\|_2
			\leq q_{\rm c}^{\,k}\|x^{(0)}-x^*\|_2.
		\end{equation}
		\item If the extreme eigenvalues are known, the optimal constant step
		size and its contraction factor are
		\begin{equation}\label{eq:optimal-richardson}
			\gamma_{\rm opt}
			=\frac{2}{\mu_1(A_\alpha)+\mu_n(A_\alpha)},
			\qquad
			q_{\rm opt}
			=\frac{S(A_\alpha)}
			{\mu_1(A_\alpha)+\mu_n(A_\alpha)}.
		\end{equation}
		Moreover, the graph invariants in \eqref{eq:application-parameters}
		give the computable limitation
		\begin{equation}\label{eq:optimal-limitation}
			q_{\rm opt}\geq\frac{L}{2U-L}.
		\end{equation}
	\end{enumerate}
\end{theorem}

\begin{proof}
	As observed in the proof of Theorem~\ref{Thm1},
	\[
	A_\alpha=(2\alpha-1)D+(1-\alpha)(D+H).
	\]
	Here $D$ is positive definite and $D+H$ is positive semidefinite.
	Consequently, $A_\alpha$ is positive definite and
	\[
	0<\underline{\lambda}\leq\mu_1(A_\alpha)
	\leq\mu_n(A_\alpha)\leq U.
	\]
	This proves the existence and uniqueness assertion. If
	$e^{(k)}=x^{(k)}-x^*$, then
	\[
	e^{(k+1)}=(I-\gamma A_\alpha)e^{(k)}.
	\]
	Since $A_\alpha$ is Hermitian,
	\[
	\|I-\gamma A_\alpha\|_2
	=\max_{1\leq i\leq n}|1-\gamma\mu_i(A_\alpha)|.
	\]
	With $\gamma=\gamma_{\rm c}$, the spectral inclusion
	$\sigma(A_\alpha)\subseteq[\underline{\lambda},U]$ gives
	\[
	\|I-\gamma_{\rm c}A_\alpha\|_2
	\leq\max\{|1-\gamma_{\rm c}\underline{\lambda}|,
	|1-\gamma_{\rm c}U|\}=q_{\rm c}.
	\]
	Iteration of this inequality proves \eqref{eq:certified-rate}.
	
	For a positive-definite Hermitian matrix, the function
	$\max_i|1-\gamma\mu_i(A_\alpha)|$ is minimized when the errors at the
	two spectral endpoints have equal magnitudes and opposite signs. Thus
	\[
	1-\gamma_{\rm opt}\mu_1(A_\alpha)
	=-\bigl(1-\gamma_{\rm opt}\mu_n(A_\alpha)\bigr),
	\]
	which gives \eqref{eq:optimal-richardson}. Finally, Theorem
	\ref{thm:local-compression-spread} gives $S(A_\alpha)\geq L$, whereas
	$\mu_n(A_\alpha)\leq U$. Therefore,
	\[
	q_{\rm opt}
	=\frac{S(A_\alpha)}{2\mu_n(A_\alpha)-S(A_\alpha)}
	\geq\frac{S(A_\alpha)}{2U-S(A_\alpha)}
	\geq\frac{L}{2U-L},
	\]
	because $s\mapsto s/(2U-s)$ is increasing for $0\leq s<2U$.
\end{proof}

The same certified step size controls the sensitivity of a residual
message-passing layer. This gives a direct connection with spectral graph
neural networks for directed data.

\begin{corollary}\label{cor:gnn-stability}
	Under the assumptions of Theorem~\ref{thm:richardson-application}, let
	$X\in\mathbb{C}^{n\times p}$ be a node-feature matrix, let
	$W\in\mathbb{C}^{p\times q}$ be a trainable weight matrix, and let
	$B\in\mathbb{C}^{n\times q}$. Suppose that
	$\sigma:\mathbb{C}^{n\times q}\to\mathbb{C}^{n\times q}$ is
	$1$-Lipschitz with respect to the Frobenius norm. Define
	\begin{equation}\label{eq:gnn-layer}
		\mathcal{F}(X)
		=\sigma\!\left((I-\gamma_{\rm c}A_\alpha)XW+B\right).
	\end{equation}
	Then, for all $X,Y\in\mathbb{C}^{n\times p}$,
	\begin{equation}\label{eq:gnn-lipschitz}
		\|\mathcal{F}(X)-\mathcal{F}(Y)\|_F
		\leq q_{\rm c}\|W\|_2\|X-Y\|_F.
	\end{equation}
	Consequently, the layer is nonexpansive whenever
	$q_{\rm c}\|W\|_2\leq1$ and contractive whenever
	$q_{\rm c}\|W\|_2<1$. More generally, a composition of $s$ such layers,
	with weight matrices $W_0,\ldots,W_{s-1}$ and $1$-Lipschitz
	activations, satisfies
	\begin{equation}\label{eq:deep-gnn-stability}
		\|X^{(s)}-Y^{(s)}\|_F
		\leq
		\prod_{\ell=0}^{s-1}\bigl(q_{\rm c}\|W_\ell\|_2\bigr)
		\|X^{(0)}-Y^{(0)}\|_F.
	\end{equation}
\end{corollary}

\begin{proof}
	By the $1$-Lipschitz property of $\sigma$, submultiplicativity of the
	Frobenius norm, and Theorem~\ref{thm:richardson-application},
	\begin{align*}
		\|\mathcal{F}(X)-\mathcal{F}(Y)\|_F
		&\leq\|(I-\gamma_{\rm c}A_\alpha)(X-Y)W\|_F\\
		&\leq\|I-\gamma_{\rm c}A_\alpha\|_2
		\|X-Y\|_F\|W\|_2\\
		&\leq q_{\rm c}\|W\|_2\|X-Y\|_F.
	\end{align*}
	This proves \eqref{eq:gnn-lipschitz}. Applying this estimate successively
	to the layers proves \eqref{eq:deep-gnn-stability}.
\end{proof}

\begin{remark}
	Corollary~\ref{cor:gnn-stability} provides a graph-computable robustness
	certificate: perturbations of the input features cannot be amplified when
	$q_{\rm c}\|W\|_2\leq1$. Conversely, if the product in
	\eqref{eq:deep-gnn-stability} decays rapidly with depth, distinct feature
	representations are forced closer together. The estimate therefore also
	quantifies a sufficient mechanism for loss of feature distinguishability in
	deep residual message passing. No claim about predictive accuracy is made;
	the conclusion is a deterministic stability guarantee.
\end{remark}

\begin{remark}
	The two conclusions have complementary meanings. Inequality
	\eqref{eq:certified-rate} supplies a stable step size and an a priori
	convergence guarantee without computing any eigenvalue. In contrast,
	\eqref{eq:optimal-limitation} shows that the locally induced spectral
	spread prevents any constant-step Richardson filter from having an
	arbitrarily small worst-case contraction factor. Moreover, to ensure
	\[
	\|x^{(k)}-x^*\|_2\leq
	\varepsilon\|x^{(0)}-x^*\|_2,\qquad 0<\varepsilon<1,
	\]
	it is sufficient to take
	\[
	k\geq\left\lceil\frac{\log\varepsilon}{\log q_{\rm c}}\right\rceil.
	\]
	For a sparse mixed graph, the nodewise rule \eqref{eq:node-update} requires
	$O(d_i)$ work at vertex $v_i$ and $O(n+m)$ arithmetic operations for one
	complete network update.
\end{remark}

\begin{example}\label{ex:directed-star-application}
	Consider the directed star $\vec S_6$ whose five arcs point from the central
	vertex to the leaves, and take $\alpha=0.7$. The Hermitian adjacency matrix
	of $\vec S_6$ is diagonally unitarily similar to the adjacency matrix of
	$K_{1,5}$; hence the two corresponding $A_\alpha$-matrices have the same
	spectrum. Here
	$
	m=5,\quad \delta=1,\quad \Delta=5,\quad
	Z_1=30,\quad F=130,\qquad t_H=0.
$
	The two nonconstant eigenvalues are
	$
	\mu_1=0.5476,\quad \mu_n=3.6524,
	$
	and the remaining four eigenvalues are equal to $0.7$. At the central
	vertex, $d_i=5$, $q_i=5$, and $\tau_i=0$, so the local compression bound is
	exact:
	\[
	L=\sqrt{(3.5-0.7)^2+4(0.3)^2(5)}
	=\sqrt{9.64}=3.1048=S(A_{0.7}).
	\]
	The cubic in Theorem~\ref{Thm1} has largest real zero
	$U=3.6528$, while $\underline{\lambda}=0.4$. Consequently,
	\[
	\gamma_{\rm c}=0.4935,\qquad q_{\rm c}=0.8026,
	\]
	and \eqref{eq:certified-rate} gives a completely graph-computable
	convergence certificate. The exact optimal quantities are
	\[
	\gamma_{\rm opt}=0.4762,\qquad q_{\rm opt}=0.7392.
	\]
	On the other hand, the new limitation \eqref{eq:optimal-limitation} gives
	\[
	q_{\rm opt}\geq\frac{3.1048}{2(3.6528)-3.1048}
	=0.7391.
	\]
	Thus, in this directed-network example, the proposed spectral estimates
	nearly recover the exact optimal contraction factor without diagonalizing
	$A_{0.7}$.\\
	The same graph also illustrates Corollary~\ref{cor:gnn-stability}.
	For the residual graph-learning layer \eqref{eq:gnn-layer}, the computed
	value $q_{\rm c}=0.8026$ gives
	\[
	\|\mathcal{F}(X)-\mathcal{F}(Y)\|_F
	\leq0.8026\|W\|_2\|X-Y\|_F.
	\]
	Thus, the layer is certified to be nonexpansive whenever
	$\|W\|_2\leq1/0.8026=1.2450$. For example, if
	$\|W\|_2\leq0.9$, then its Lipschitz constant is at most
	$0.8026(0.9)=0.7224$. Consequently, a composition of five such
	layers satisfies
	\[
	\|X^{(5)}-Y^{(5)}\|_F
	\leq(0.7224)^5\|X^{(0)}-Y^{(0)}\|_F
	=0.1968\|X^{(0)}-Y^{(0)}\|_F.
	\]
	Hence, this graph-computable certificate guarantees that the five-layer
	mapping reduces input-feature perturbations to at most approximately
	$19.68\%$ of their original magnitude.
\end{example}

\section{Conclusion}
We derived spectral-moment and vertexwise bounds for the extreme eigenvalues of Hermitian $A_\alpha$-matrices. The diagonal spectral-enclosure argument produced a degree-sequence refinement whose specialization to ordinary graphs strictly improves an earlier degree-based spread bound. A local two-dimensional compression further yielded a spread estimate involving neighbouring degrees and gain-weighted triangles; this estimate is exact for every mixed orientation of a star and is independent of the known Zagreb-index bound. For sums of the smallest eigenvalues, two complementary estimates were obtained from aggregate and pairwise principal compressions. Finally, the applications to Richardson graph filtering and residual graph-learning layers showed that the spectral bounds provide a computable convergence certificate, a lower bound on the optimal contraction factor, and a layerwise Lipschitz stability guarantee for directed-network computations. In the directed-star example, the contraction-factor bound is numerically close to the exact value.
\section*{Acknowledgments}
Both authors gratefully acknowledge support from the Anusandhan National Research Foundation (ANRF), Government of India, under Grant No. ANRF/IRG/2025/000280/MS.
\section*{Declarations}
\textbf{Conflict of interest.} The authors declare that they have no conflict of interest.\\
\textbf{Data availability.} No datasets were generated or analysed during the current study.

\end{document}